\documentclass{amsart}
\usepackage{xcolor}
\usepackage{amssymb,latexsym,amsmath,extarrows}
\usepackage{graphicx,mathrsfs}

\newtheorem{theorem}{Theorem}[section]
\newtheorem{lemma}[theorem]{Lemma}
\newtheorem{proposition}[theorem]{Proposition}

\newtheorem{definition}[theorem]{Definition}

\newtheorem{conjecture}[theorem]{Conjecture}

\newcommand{\e}{\varepsilon}

\newcommand{\f}{\frac}

\newcommand{\wh}{\widehat}

\newcommand{\ti}{\tilde}
\newcommand{\Id}{{\bf 1}}

\newcommand{\bT}{{\bf T}}

\begin{document}

\title{A Stein-Tomas type estimate and a decoupling inequality }
\author{Xiaochun Li}

\address{
Xiaochun Li\\
Department of Mathematics\\
University of Illinois at Urbana-Champaign\\
Urbana, IL, 61801, USA}

\email{xcli@math.uiuc.edu}

\date{}

\begin{abstract}
 A Stein-Tomas type inequality and a (weak) decoupling inequality are proved by using the polynomial partitioning method.  Both estimates are related closely to Waring's problem. 
\end{abstract}
\maketitle

\section{Introduction}
\setcounter{equation}0

Let $N\in\mathbb N$ and $d$ be a real number $\geq 3$. We chose $N$ many points, say $P_1, \cdots, P_N$, which are equally distributed in the curve
$\Gamma:=\{(\xi, \xi^d)\in\mathbb R^2: 1\leq |\xi|\leq 2\}$.  Let us explain the precise meaning of the equally distributed points. 
Suppose that $P_j=(\xi_j, \xi_j^d)\in \Gamma$ for each $j=1, \cdots, N$.  We say $P_j$'s are equally distributed in $\Gamma$ if 
$\{\xi_1, \cdots, \xi_N\}$ forms an arithmetic progression with common difference $N^{-1}$.   For any point $P\in \Gamma$, let $\omega_{P}$ denote an $N^{-1}\times N^{-d}$-rectangle, centered at $P$,  with its long side parallel to the tangent line of the curve $\Gamma$ at the point $P$.  We set $\Omega_N$ to be 
$$
\Omega_N:=   \big\{ \omega_{P_1}, \cdots, \omega_{P_N} \big\}\,.
$$
We use $\Id_E$ to denote the characteristic function of a measurable set $E\subset \mathbb R^2$. 
For any $\omega\in \Omega_N$, $f_\omega$ represents Fourier restriction of $f$ in the $\omega$, that is, $\wh{f_\omega}
= \wh f \Id_{\omega}$.  \\


The main theorem we want to establish is 

\begin{theorem}\label{thm1}
Let $p\geq 2d+2$. For any $\e>0$, there is a constant $C_{p,\e}$ depending on $p$ and $\e$ such that 
\begin{equation}\label{eq1}
\big\| \sum_{\omega\in \Omega_N} f_\omega \big\|_{L^p( B^2(N^d))}\leq C_{p, \e} N^{-\frac{d}{2}+\e}\|f\|_2\,,
\end{equation}
for any $f\in L^2$ and any $N\in \mathbb N$.  Here $B^2(N^d)$ stands for a ball in $\mathbb R^2$, of radius $N^d$. 
\end{theorem}

Notice that when $d=2$, (\ref{eq1}) becomes the well-known Stein-Tomas inequality associated to an $N^{-2}$-neighborhood of a parabola. 
Hence, Theorem \ref{thm1} is a Stein-Tomas type restriction estimate for a function whose Fourier transform is supported in a broken line, instead of a smooth curve with a positive Gaussian curvature.  
 For $d>2$, those thin rectangles $\omega_{P_j}$'s no longer lie in the $N^{-d}$-neighborhood of the curve $\Gamma$.  This is the real enemy, which invalidates 
the $TT^*$-method used in the proof of the Stein-Tomas theorem.  In addition,   a simple example involving a single wavepacket shows that  (\ref{eq1}) is nearly sharp (up to $N^\e$) and the $p=2d+2$ is the smallest number to maintain the decay of $N^{-d/2}$ in the right side. \\

Theorem \ref{thm1} was motivated by the study of the following mean value problem.
\begin{conjecture}\label{conj1}
Let $d\geq 3$ be an integer. 
$\forall \e>0$ and $p\geq 2(d+1)$,
\begin{equation}\label{eqn2}
 \bigg\|\sum_{n=1}^N a_n e^{2\pi i x_1 n + 2\pi i x_2 n^d}  \bigg\|_{L^p((\mathbb T^2)} \lesssim 
 N^\e \bigg( \sum_{n=1}^N\big|a_n\big|^2\bigg)^{\f12}\,. 
 \end{equation}
 \end{conjecture}

Conjecture \ref{conj1} is certainly very interesting in additive number theory because it leads to a large improvement on Waring's problem.  It is also closely related to Fourier analysis, for instance,  
(\ref{eqn2}) can be viewed as a discrete Fourier restriction problem on the curve $(n, n^d)$.  
Moreover,  a standard deduction leads the conjecture to the following decoupling problem. 

\begin{conjecture}\label{conj2}
Let $P_j=(\xi_1, \xi_j^d)$ ($j=1, \cdots, N$) be equally distributed in $\Gamma$. For any $\e>0$ and 
$p\geq 2(d+1)$, 
\begin{equation} \label{eqn3}
\big\| \sum_{\omega\in\Omega_N} f_\omega\big\|_{p}\lesssim N^{\f12-\frac{d+1}{p}+\e} \bigg( \sum_{\omega\in\Omega_N}
  \big\| f_\omega\big\|_p^2 \bigg)^{\f12}\,.
\end{equation}
\end{conjecture}

The case when $d=2$ was solved by Bourgain and Demeter \cite{BD}.  For $d>2$, it is not clear that if
(\ref{eqn3}) is true, although it is not difficult to see the sharpness of $(\ref{eqn3})$, up to a factor of $N^\e$. 
Unlike $(\xi, \xi^2)$, $(\xi, \xi^d)$ with $d>2$ is distorted badly under translations.  This seems to be the biggest obstacle. It results in a significant enemy in the induction on scales, the most crucial tool used in 
all existing methods in the decoupling theory.  We are still working on seeking a way to prove the decoupling inequalities without employing the induction on scales too much.  By using the method in the proof of Theorem \ref{thm1}, we are able to obtain some weak version of decoupling inequalities as follows. 

\begin{theorem} \label{thm2}
Let $P_j=(\xi_1, \xi_j^d)$, $j=1, \cdots, N$, be equally distributed in $\Gamma$. 
For any $\e>0$ and $p\geq 2(d+1)$, 
\begin{equation}
\big\| \sum_{\omega\in\Omega_N} f_\omega\big\|_{p}\lesssim 
N^{\f12-\frac{d+1}{p}+\e}  N^{-\frac{d+1}{p}}\bigg(\sum_{\omega\in\Omega_N} \big\|f_\omega\big\|_{p/2}^2\bigg)^{\f12}\,.
\end{equation}
\end{theorem}

Here we only state Theorem \ref{thm2}, without presenting its proof in details because it can be justified 
as  Theorem \ref{thm1}.  Let us outline the ideas used in the proof of Theorem \ref{thm1}.  
We employ the polynomial partitioning method  in \cite{Guth} developed by Guth.  For a polynomial $P$ in the polynomial ring $\mathbb R[x]$, $Z(P)$ denotes the variety $\{ x\in\mathbb R^n: P(x)=0\}$.   
Using a variety $Z(P)$ generated by a polynomial $P$ of degree $D$, the $N^d$-ball can be divided in to $O(D^2)$ many cells $O_j$ such that 
$$\int_{O_j}\big| \sum_\omega f_\omega\big|^p \sim \mu\,.$$
Here $\mu$ is a dyadic number independent of $j$. In order to taking wavepackets into account,  each $O_j$ 
is replaced by a smaller cell $O_j'$, given by $O_j\backslash W$, where $W$ is an $N$-neighborhood of the variety $Z(P)$.  The cellular contribution from $O_j'$'s can be treated by the induction on scales. Henceforth, 
it remains to handle the wall contribution 
$$
\int_W \big| \sum_\omega f_\omega\big|^p\,, 
$$
which can be further broken into transverse and tangential parts.  The tangential part is related to those rectangles contained in a small neighborhood of the variety $Z(P)$, and thus we run into a simple $1$-dimensional problem.  The transverse part usually can be removed by the induction on scales. However, 
because the thin $N\times N^d$-rectangles do not stay in an $N^{-d}$-neighborhood of the curve $\Gamma$,
the original treatment in \cite{Guth} for the transverse contribution can not be applied. This is the main enemy we encounter.   To overcome this difficulty, we have to consider more delicate generalized versions of the transverse and tangential rectangles, classified by the (dyadic) magnitude of angle formed by the direction of the rectangle and the variety.   When $p\geq 8$, 
the principal contribution comes from only two extremal cases, the most tangential rectangles which 
lie completely in an $N$-neighborhood of the variety $Z(P)$ and the most transverse rectangles which 
cross the variety $Z(P)$ by an angle $\sim 1$.  Because of this novel phenomenon,  the transverse contribution can be treated directly without using the induction on scales. Therefore the aforementioned 
obstacle can be overcome and Theorem \ref{thm1} can be achieved. \\

From the proof of Theorem \ref{thm1},  we also see that if the wall contribution dominates, the decoupling
inequality (\ref{eqn3}) in Conjecture \ref{conj2} holds.  However, it is challenging to  deduce the problem to the algebraic contribution.  \\

{\bf Acknowledgement} The author is supported by Simons collaboration grants.  

\section{Proof of Theorem \ref{thm1}} 

\subsection{Results of the polynomial partitioning}
The Borsuk-Ulam theorem asserts that if $F:\mathbb S^n\rightarrow \mathbb R^n$ is a continuous function, 
where $\mathbb S^n$ is the $n$-dimensional unit sphere, then there exists a point $v\in \mathbb S^n$ with $F(v)=F(-v)$. 
This theorem yields a beautiful application in the proof of Stone-Tukey's ham sandwich theorem (see \cite{Guth}), in which 
one can use a polynomial $P$ of degree $ O(m^{1/n})$ to bisect a family of $L^1$-functions $f_1, \cdots, f_m$.  
We state some polynomial partitioning results in the continuous setting here, whose proof can be found in \cite{Guth}.\\

\begin{theorem}\label{part1}
Suppose that $F$ is a non-negative integrable function on $\mathbb R^n$. Then for each $D\in\mathbb N$ there is a non-zero
polynomial $P$ of degree at most $D$ so that $\mathbb R^n\backslash Z(P)$ is a unon of $\sim D^n$ disjoint open sets $O_j$'s and the integral $\int_{O_j} F$ are all equal. 
\end{theorem}

A polynomial $P$ is called {\it non-singular} if $\nabla P(x)\neq 0$ for each point in $Z(P)$.  The variety $Z(P)$ is a smooth
hypersurface if $P$ is non-singular.  Technically it is more convenient to work on non-singular polynomials. It is known that 
non-singular polynomials are dense. Hence,  An analogy of Theorem \ref{part1} associated to the non-singular polynomials 
can be achieved as follows. 

\begin{theorem}\label{part2}
Suppose that $F$ is a non-negative integrable function on $\mathbb R^n$. Then for each $D\in\mathbb N$ there is a non-zero
polynomial $P$ of degree at most $D$ such that 
\begin{itemize}
\item the polynomial $P$ is a product of non-singular polynomials. 
\item  $\mathbb R^n\backslash Z(P)$ is a union of $\sim D^n$ disjoint open sets $O_j$'s and the integral $\int_{O_j} F$ agree 
up to a factor of $2$. 
\end{itemize}
\end{theorem}

\subsection{ Wave packet decomposition} 
For any $R^{-d}\times R^{-1}$-rectangle $\omega\in\Omega_R$ in frequency,  $T$ is called a dual rectangle in physical space
if $T$ is an $R^d\times R$-rectangle whose direction is perpendicular to the long side of $\omega$.  We tile up the physical 
space by using a collection of disjoint dual rectangles of $\omega$, which is denoted by $\bT_\omega$.  For each dual rectangle $T\in \bT_\omega$, we associate it with a bump function $\phi_T: \mathbb R^2\rightarrow \mathbb C$, obeying 
$\phi_T$ is a non-negative Schwartz function whose Fourier transform is supported in the translate of $\omega$ at the origin, $\phi_T$ is essentially supported on $T$ up to a Schwartz tail, 
and 
\begin{equation}
\sum_{T\in \bT_\omega}\phi_T(x)=1\,,\,\,  \forall x\in \mathbb R^2.
\end{equation}
Because $\phi_T$'s form a partition of unity, we can represent $f_\omega$ as 
\begin{equation}\label{wave}
f_\omega = \sum_{T\in\bT_\omega} f_\omega \phi_T\,. 
\end{equation}
The function $f_\omega\phi_T$ is called a wave-packet. It is well-localized in both physical and frequency spaces, since
its Fourier transform is supported in $2\omega$ and in physical space it is essentially localized in $T$ up to a Schwartz tail. 
Second, when localized in $T$, the size of the wavepacket $f_\omega\phi_T$ can be viewed as a constant.  
Let $\delta=\e^{10}$, a number much smaller than $\e$.  We use $T^*$ to 
denote $R^{\delta}T$.  $T^*$ is designed to handle the Schwartz tail contribution from $\phi_T$. It has no significant difference from $T$.  By (\ref{wave}), we can make the wave packet decomposition 
\begin{equation}
\sum_{\omega\in \Omega_R} f_\omega=\sum_\omega \sum_{T\in\bT_\omega} f_\omega \phi_T\,. 
\end{equation}
The pair $(\omega, T)$ is called a tile. 
For any rectangle $T\subset \mathbb R^2$, $e(T)\in\mathbb S^1$ denotes the direction of the long side of $T$. 
When $T$ is fixed, its direction $e(T)$ is determined and therefore so does the frequency location $\omega$, which can be denoted by $\omega(T)$.  By this observation we can
simplify the notation by setting, for any $T\in \cup_{\omega}\bT_\omega$, 
$$
 f_T:= f_{\omega(T)}\phi_T\,. 
$$
The function $f_T$ is essentially supported in $T\times \omega(T)$ in the physical and frequency spaces. 
We now end up with our final form of the wave packet decomposition,
\begin{equation}
 \sum_{\omega\in \Omega_R} f_\omega = \sum_{T\in\bT} f_T\,, 
\end{equation}
where $\bT= \cup_{\omega\in \Omega_R} \bT_\omega$. 

\subsection{Application of polynomial partitioning} 

Without loss of generality, we can assume $N$ in Theorem \ref{thm1} to be a dyadic number, denoted by $R$. 
We now apply Theorem \ref{part2} to the integrable  function $\big| \sum_{\omega} f_\omega\big|^p$. For any (sufficiently large)
positive integer $D$, there is a non-zero polynomial $P$ of at most degree $D$ such that
\begin{itemize}
\item the polynomial $P$ is a product of non-singular polynomials,
\item  $\mathbb R^2\backslash Z(P)$ is a union of $c D^2$ disjoint open sets $O_j$'s.
\item  $\int_{O_j} \big|\sum_\omega f_\omega \Id_{B^2(R^d)} \big|^p \sim \int_{O_{j'}}  \big|\sum_\omega f_\omega  \Id_{B^2(R^d)} \big|^p$
 for any possible $j$ and $j'$. 
\end{itemize}
To make the wave packets involved, we will make each cell $O_j$ smaller.   Let $W$ be the $R^{1+\delta}$-neighborhood of the variety $Z(P)$. Let $O'_j$ denote $O_j\backslash W$. 
Then we can write 
$$
 \int_{B^2(R^d)} \bigg| \sum_{\omega\in\Omega_R}f_\omega\bigg|^p = \sum_{j}\int_{O'_j\cap B^2(R^d)}  \bigg| \sum_{\omega\in\Omega_R}f_\omega\bigg|^p  +  \int_{W\cap B^2(R^d)}  \bigg| \sum_{\omega\in\Omega_R}f_\omega\bigg|^p  \,.$$
The first term here is the cell contribution and the second term is the wall contribution.  
We can assume that $R$ is much larger than $D$ otherwise the desired conclusion would be trivial. For the wall contribution, we have the following estimates. 
\begin{proposition}\label{prop-wall}
For $p\geq 2d+2$, 
\begin{equation}\label{w11-est}
\bigg( \int_{W\cap B^2(R^d)}  \bigg| \sum_{\omega\in\Omega_R}f_\omega\bigg|^p \bigg)^{\frac{1}{p}}
\lesssim R^{-\frac{d}{2}+ 3\delta} \|f\|_2\,. 
\end{equation}
\end{proposition}
This is the most technical result in the paper, and it follows immediately from Proposition \ref{prop-alg}, stated and proved in 
Section \ref{alg-sec}. With (\ref{w11-est}) in hand, we are able to finish the proof of Theorem \ref{thm1} in a standard way. 
If the wall term dominates, we are done by Proposition \ref{prop-wall}.  Thus we can assume that 
\begin{equation}
 \int_{B^2(R^d)} \bigg| \sum_{\omega\in\Omega_R}f_\omega\bigg|^p \leq 10\sum_{j\in J}\int_{O'_j\cap B^2(R^d)}  \bigg| \sum_{\omega\in\Omega_R}f_\omega\bigg|^p\,,
 \end{equation}
 and for $j\in J$, $ \int_{O'_j\cap B^2(R^d)}  | \sum_{\omega\in\Omega_R}f_\omega|^p$
 agree up to  a factor of $10$. 
Taking the wave packets into account,  the worst scenario occurs if 
\begin{equation}\label{wst}
\int_{B^2(R^d)} \bigg| \sum_{\omega\in\Omega_R}f_\omega\bigg|^p \leq 10\sum_{j\in J'}\int_{O'_j\cap B^2(R^d)}  \bigg| \sum_{T\in\bT[O'_j]}f_T\bigg|^p\,,
\end{equation}
where $\bT[O'_j]$ stands for $\{T\in\bT:  T^*\cap O'_j\neq \emptyset\}$.  This is because that those $T$'s with 
$T^*\cap O'_j=\emptyset$ contribute a rapid decay of $R$.  By pigeonholing, we can assume that 
for $j\in J'\subset J$,  those $\int_{O'_j\cap B^2(R^d)}  | \sum_{T\in\bT[O'_j]}f_T|^p $ agree up to a factor $100$. 
 In addition, we can assume that $J$ has a cardinality at least 
$\frac{c}{2^{10}}D^2 $, otherwise,   up to a rapid decay contribution, 
 the right side of (\ref{wst}) is at most a small portion of the left one, which leads to the desired estimates in Theorem \ref{thm1}. Moreover, using $L^2$-orthogonality, we see that 
 \begin{equation}\label{ort}
\sum_{j\in J'} \big\|\sum_{T\in \bT[O'_j]}f_T\big\|_2^2 \leq \sum_{j\in J'} \sum_{T\in \bT[O'_j]}\big\|f_T\big\|_2^2 
\leq \sum_{T\in \bT} \sum_{\substack{j\in J'\\ T\cap O'_j\neq \emptyset}} \|f_T\|_2^2\lesssim 
D\|f\|_2\,. 
\end{equation}
In the last step of (\ref{ort}), we used an important geometric fact asserting that any rectangle $T$ enters at most $O(D)$ many 
cells. This serves a base of the polynomial partitioning method and it is a consequence of the fundamental theorem of Algebra. 
 Because there are $\sim D^2$ many $j\in J'$, from the pigeonhole principle, there exists a $j_0\in J'$ such that 
 \begin{equation}\label{j-0}
 \big\|\sum_{T\in \bT[O'_{j_0}]}f_T\big\|_2\lesssim D^{-1/2}\|f\|_2\,. 
  \end{equation}

Define $\beta\in \mathbb R$ to be the best constant such that 
 \begin{equation}
 \big\| \sum_{\omega\in\Omega_R}f_\omega\big\|_{B^2(R^d)} \leq R^\beta \|f\|_2\,,
 \end{equation}
 holds for any (dyadic) $R\geq 2$.   \\

 Employing the uniform distribution property of those integrals in the right side of (\ref{wst}), (\ref{j-0}) and the definition of $\beta$, we see that 
 \begin{equation}\label{f}
 \big\|\sum_{\omega\in\Omega_R}f_\omega \big\|_{L^p(B^2(R^d))}\lesssim D^\frac{2}{p}
  \big\|\sum_{T\in\bT[O'_{j_0}]}f_T \big\|_{L^p(B^2(R^d))}\lesssim D^{\frac{2}{p}-\f12} R^\beta \|f\|_2\,.  
 \end{equation}
The number $D$ can be chosen to be sufficiently large so that the right side of (\ref{f}) is much less than $R^{\beta-}\|f\|_2$.
Then we see that the cellular term can be handled by the induction on scales. Therefore, the wall term dominates and now Theorem 
\ref{thm1} follows from Proposition \ref{prop-wall}.

\section{Geometric lemmas}\label{ge}

Let $L$ and $W$ be positive numbers such that $L>W$.  Let $\mathcal T$ be a collection of $L\times W$-rectangles $T$, obeying $e(T)$'s are evenly distributed in $\mathbb S^1$ and the cardinality of $\{e(T): T\in\mathcal T\}$ is $J$. 
Let $\mathcal T'\subset \mathcal T$ such that 
\begin{itemize}
\item   distinct rectangles in $ \mathcal T'$ point to distinct directions, 
\item $\{ e(T): T\in \mathcal T'\}=\{e(T): T\in \mathcal T\}$. 
\end{itemize} 
 We use $|E|$ to denote Lebesgue measure of $E\subset \mathbb R^2$.  

\begin{lemma}\label{lemLW}
Suppose that $J\leq 10L/W$. Then, 
for any $T\in \mathcal T'$, 
\begin{equation}\label{est-ST}
 \sum_{S\in \mathcal T'}\big | S\cap T\big|\leq  \log\big(J\big) |T|\,. 
\end{equation}
\end{lemma}

\begin{proof}
For any $S\in \mathcal T'$, the angle formed by $e(S)$ and $e(T)$ is $j/J$ for some $j\in \{1, \cdots,J\}$. We can estimate the left side of (\ref{est-ST}) by
$$
\sum_{j=1}^{J} \frac{W^2}{j /J} = \sum_{j=1}^{J} \frac{1}{j}  JW^2 \lesssim \log\big(J\big) |T|\,,
$$
since $J\lesssim L/W$. 
\end{proof}

\begin{lemma}\label{lem-L-1T}
Suppose that $J\leq 10L/W$. Then 
\begin{equation}
\sum_{T\in \mathcal T'} |T|\leq \log\big(J\big) \bigg| \bigcup_{T\in \mathcal T'} T\bigg|\,.
\end{equation}
\end{lemma}

\begin{proof}
By Cauchy-Schwarz inequality, we have 
\begin{equation}\label{CS}
 \sum_{T\in \mathcal T'} |T| \leq \big\| \sum_{T}\Id_T\big\|_2  \bigg| \bigcup_{T\in \mathcal T'} T\bigg|^{1/2}\,.
  \end{equation}
On the other hand, by Lemma \ref{lemLW}, we see that 
\begin{equation}\label{sin}
 \big\| \sum_{T}\Id_T\big\|^2_2\leq \sum_{T}\sum_S\big|S\cap T\big|\leq  \log\big(J\big) \sum_T|T|\,. 
  \end{equation}  
The desired inequality (\ref{lem-L-1T}) now follows from (\ref{CS}) and (\ref{sin}).

\end{proof}

We need a theorem of Wonggkew \cite {Won}, describing quantitatively the volumes of neighborhoods of real algebraic varieties.

\begin{theorem}[Wongkew]\label{won}
Let $P$ be a non-zero polynomial of degree $D$ on $\mathbb R^n$. Then 
$$
 \big|  B^n(L) \cap {\mathcal N}_{\rho} Z(P) \big| \leq C_n D \rho L^{n-1}\,
$$
where $B^n(L)$ is an $L$-ball in $\mathbb R^n$ and ${\mathcal N}_{\rho} Z(P)$ is a $\rho$-neighborhood of 
the variety $Z(P)$. 
\end{theorem}

\begin{lemma}\label{lem-geo}
Let $\mathcal T$ be the collection of rectangles described in the beginning of Section \ref{ge}.
Suppose that $J\leq 10L/W$ and every rectangle $T\in \mathcal T$ is contained in $B^2(K_1L)\cap \mathcal N_{K_2W}Z(P)$,
where $P$ is a non-zero polynomial of degree $D$ on $\mathbb R^2$. 
then the cardinality of $ \{e(T): T\in \mathcal T\}$ is at most $ cDK_1K_2\log(J)$, where $c$ is an absolute constant and $K_1, K_2\geq 1$. 
\end{lemma}

\begin{proof}
 Using Lemma \ref{lem-L-1T} and Theorem \ref{won},  we see that 
\begin{equation}\label{w1}
 \sum_{T\in \mathcal T'} |T| \leq  \log\big(J\big)  \big|B^2(K_1L)\cap \mathcal N_{K_2W}Z(P)\big|
 \lesssim   DK_1K_2(LW) \log\big(J\big)\,.
 \end{equation}
  Notice that $|T|=LW$ and $\sum_{T\in \mathcal R[L, W]}$ equals to the cardinality of the set $\mathcal R[L, W]$,
  then we see that (\ref{w1}) implies the desired upper bound $cDK_1K_2\log(J)$ for the cardinality of $ \{e(T): T\in \mathcal T\}$. 
\end{proof}

\section{Transverse and tangential rectangles}

For any $k\in\mathbb Z$,  let $\mathcal Q_{k}$ denote a set consisting of all dyadic cubes in $\mathbb R^2$, whose side-length is $2^k$.  Let $k^*$ satisfy $R^d\leq 2^{k^*}\leq  2 R^{d}$. Without loss of generality, we can assume that  the ball $B^2(R^d)$ is contained in a dyadic 
  $2^{k^*}$-cube $Q^{*}\in\mathcal \mathcal Q_{k^*}$, which is the maximal dyadic cube we need to take into account. 

For any dyadic number $\Xi\in [1, R^{d-2}]$,  we decompose the dyadic cube $Q^*$ into $2^{k^*}/\Xi$-cubes in 
$\mathcal Q_{k^*-\log_2\Xi}$.  We use ${\mathcal Q}^*_{\Xi}$ to denote the set of those $2^{k^*}/\Xi$-sub-cubes of $Q^*$.  

\begin{definition}\label{def-X}
For any $Q\in {\mathcal Q}^*_{\Xi}$,  $\bT_{\Xi}[Q]$ is a set consisting of all $T\in \bT$ such that 
\begin{itemize}
\item  $T^*\cap W\cap Q\neq \emptyset $
\item  If $z$ is any non-singular point of $Z(P)$ lying in $3Q\cap 5T^*$, then 
$${\rm Angle} \big(e(T), T_z\big[Z(P)\big]\big) \leq  \Xi  R^{-d+1+\delta} \,.$$
Recall that $e(T)$ is the unit vector in the direction of the rectangle $T$, and $T_z[Z(P)]$ stands for the tangent space to the variety $Z(P)$ at the point $z$. 
\end{itemize}
\end{definition}

From this definitions, we see immediately that if $T$ satisfies $T^*\cap W \cap Q\neq \emptyset$ and $T\notin \bT_{\Xi}[Q]$, 
then there is a non-singular point $z$ of $Z(P)$ lying in $3Q\cap 5T^*$ such that
\begin{equation}\label{an-Xi}
 {\rm Angle} \big(e(T), T_z\big[Z(P)\big]\big)  > \Xi  R^{-d+1+\delta} \,.
 \end{equation}
 
\vspace{0.5cm} 
 
Let $\Delta\in [R^\delta/2, R/2^5]$ be any dyadic number. For any such a dyadic number, we partition the dyadic cube $Q^*$ into 
$\frac{2^{k^*}}{R^{d-2}\Delta}$-cubes in $ \mathcal Q_{k^*-\log_2(R^{d-2}\Delta)}$.  To simplify the notations, we set
${\mathcal Q}^{**}_\Delta$ to be the set consisting of all those $\frac{2^{k^*}}{R^{d-2}\Delta}$-cubes lying in $Q^*$. 

\begin{definition}\label{def-D}
For any $Q\in {\mathcal Q}^{**}_\Delta$,  $\bT_\Delta[Q]$ is a set 
consisting of all $T\in \bT$ such that 
\begin{itemize}
\item  $T^*\cap W\cap Q\neq \emptyset $
\item  If $z$ is any non-singular point of $Z(P)$ lying in $3Q\cap 5T^*$, then 
$${\rm Angle} \big(e(T), T_z\big[Z(P)\big]\big) \leq  \Delta R^{-1} \,.$$
\end{itemize}
\end{definition}

From Definition \ref{def-D},  it follows that if $T$ satisfies $T^*\cap W \cap Q\neq \emptyset$ and $T\notin \bT_{\Delta}[Q]$, 
then there is a non-singular point $z$ of $Z(P)$ lying in $3Q\cap 5T^*$ such that
\begin{equation}\label{an-Del}
 {\rm Angle} \big(e(T), T_z\big[Z(P)\big]\big)  > \Delta R^{-1} \,.
 \end{equation}
We will use this inequality in our proof.  When (\ref{an-Del}) holds, the rectangle $T$ is transverse to $Z(P)$ at some point 
$z\in 3Q$ with an angle at least $\Delta R^{-1}$.  \\

\begin{lemma}\label{lem-q-q'}
1) Let $Q'\in \mathcal Q^*_{\Xi}$, $Q\in \mathcal Q^*_{\Xi/2}$ and $Q'\subseteq Q$. 
If $T^*\cap W\cap Q'\neq \emptyset$, then 
$ T\notin \bT_{\Xi}[Q']$ implies $T\notin \bT_{\Xi/2}[Q]$. \\
2) Let $Q'\in \mathcal Q^{**}_{\Delta}$, $Q\in \mathcal Q^*_{\Delta/2}$ and $Q'\subseteq Q$. 
If $T^*\cap W\cap Q'\neq \emptyset$, then 
$ T\notin \bT_{\Delta}[Q']$ implies $T\notin \bT_{\Delta/2}[Q]$. 
\end{lemma}

\begin{proof}
We only present a proof for Part 1). Part 2) can be proved similarly.  Since $T^*\cap W\cap Q'\neq \emptyset$, 
by Definition \ref{def-X}, we see that there is a non-singular point $z$ of $Z(P)$ lying in $3Q'\cap 5T^*$ such that
\begin{equation}\label{an-Xi1}
 {\rm Angle} \big(e(T), T_z\big[Z(P)\big]\big)  > \Xi  R^{-d+1+\delta} \,.
 \end{equation}
Clearly it yields that $T\notin  \bT_{\Xi/2}[Q]$. 
\end{proof}

\begin{lemma}\label{lem-ind}
For $\Xi\geq 2$, 
\begin{eqnarray*}
 & & \bigg( \sum_{Q\in \mathcal Q^*_{\Xi/2}}  \int_{W\cap Q} \big| \sum_{T\notin \bT_{\Xi/2}[Q]}f_T\big|^p \bigg)^{\frac{1}{p}}\\
 & \leq & 
 \bigg( \sum_{Q'\in \mathcal Q^*_{\Xi}}  \int_{W\cap Q'} \big| \sum_{T\in \bT_{\Xi}[Q']\backslash \bT_{\Xi/2}[\ti Q']}f_T\big|^p \bigg)^{\frac{1}{p}}\\
& &    +  \bigg( \sum_{Q'\in \mathcal Q^*_{\Xi}}  \int_{W\cap Q'} \big| \sum_{T\notin \bT_{\Xi}[Q']} f_T\big|^p \bigg)^{\frac{1}{p}} + C_\delta R^{-100/\delta^{10}}\|f\|_2\,.
  \end{eqnarray*}
Here $\ti Q$ denotes the unique parent of the dyadic cube $Q$. 
The inequality above still holds if $\Xi$ is replaced by $\Delta$. 
\end{lemma}

\begin{proof}
Since $Q$'s are disjoint, the left side of the inequality in the lemma is controlled by 
$$
\bigg( \int_{W} \bigg|\sum_{Q\in \mathcal Q^*_{\Xi/2}}   \sum_{T\notin \bT_{\Xi/2}[Q]}f_T\Id_Q \bigg|^p \bigg)^{\frac{1}{p}}
= \bigg( \int_{W} \bigg|\sum_{Q\in \mathcal Q^*_{\Xi/2}} \sum_{\substack{Q'\in \mathcal Q^*_{\Xi} \\Q'\subset Q}}  \sum_{T\notin \bT_{\Xi/2}[Q]}f_T\Id_{Q'} \bigg|^p \bigg)^{\frac{1}{p}}\,,
$$
which can be split into two parts, by Minkowski's inequality, 
$$
 \bigg( \!\int_{W} \bigg|\!\!\sum_{Q\in \mathcal Q^*_{\Xi/2}} \sum_{\substack{Q'\in \mathcal Q^*_{\Xi} \\Q'\subset Q}}  \sum_{ \substack{T\notin \bT_{\Xi/2}[Q] \\ T\in \bT_\Xi[Q']}}f_T\Id_{Q'} \bigg|^p \bigg)^{\frac{1}{p}}
\!\!+
 \bigg( \!\int_{W}\!\!\bigg|\sum_{Q\in \mathcal Q^*_{\Xi/2}} \sum_{\substack{Q'\in \mathcal Q^*_{\Xi} \\Q'\subset Q}}  \sum_{\substack{T\notin \bT_{\Xi/2}[Q] \\ T\notin \bT_\Xi[Q']}}f_T\Id_{Q'} \bigg|^p \bigg)^{\frac{1}{p}}\,.
$$
In the first term above, $Q$ is the parent of $Q'$. Thus by Fubini's theorem, we see that it becomes 
\begin{equation}\label{1-est0}
 \bigg( \sum_{Q'\in \mathcal Q^*_{\Xi}}\int_{W\cap Q'} \bigg| 
 \sum_{ T\in \bT_\Xi[Q']\backslash\bT_{\Xi/2}[\ti Q']}f_T\Id_{Q'} \bigg|^p \bigg)^{\frac{1}{p}}\,.
 \end{equation}
 Using Lemma \ref{lem-q-q'} and noticing the rapid decay contributions from those $T$'s with $T^*\cap W\cap Q'=\emptyset$, 
 the second term can be dominated by 
 \begin{equation}\label{2-est0}
  \bigg( \sum_{Q'\in \mathcal Q^*_{\Xi}}  \int_{W\cap Q'} \big| \sum_{T\notin \bT_{\Xi}[Q']} f_T\big|^p \bigg)^{\frac{1}{p}} + C_\delta R^{-100/\delta^{10}}\|f\|_2\,.
 \end{equation}
Then the desired inequality follows from (\ref{1-est0}) and (\ref{2-est0}).
\end{proof}

Using Definition \ref{def-X}, Definition \ref{def-D}, Lemma {\ref{lem-ind}} inductively, we can dominate 
\begin{equation}
  \bigg( \int_{W\cap Q^*}\big|  \sum_{T} f_T \big|^p\bigg)^{\frac{1}{p}} 
  \leq   {\rm Tang} + {\rm Tran} + C_\delta R^{-100/\delta} \|f\|_2\,,
 \end{equation}
 where ${\rm Tang}$ and ${\rm Tran}$ denote the contributions from the tangential rectangles and the transverse rectangles, respectively.   The tangential contribution can be expressed as 
\begin{equation}
{\rm Tang} :=
\sum_{\Xi} \bigg( \sum_{Q\in {\mathcal Q}^*_{\Xi} } \int_{W\cap Q}  \big| \sum_{T\in \bT_{\Xi, {\rm tang}}[Q] } f_T\big|^p\bigg)^{\frac{1}{p}}\,,
\end{equation}
where $\bT_{1, {\rm tang}}[Q] = \bT_{1}[Q^*]$ if $\Xi=1$,  $\bT_{\Xi, {\rm tang}}[Q] =\bT_\Xi[Q]\backslash \bT_{\Xi/2}[\ti Q]$
if $\Xi>1$. 
And the transverse contribution is given precisely by 
\begin{equation}
 {\rm Tran}:=
\sum_{\Delta} \bigg( \sum_{Q\in {\mathcal Q}^{**}_{\Delta} } \int_{W\cap Q}  \big| \sum_{T\in \bT_{\Delta, {\rm tran}}[Q] } f_T\big|^p\bigg)^{\frac{1}{p}}\,,
\end{equation}
where $\bT_{\Delta, {\rm tran}}[Q]=\bT_\Delta[Q]\backslash \bT_{\Delta/2}[\ti Q]$ if $\Delta\leq R/{2^6}$ and 
$\bT_{\Delta, {\rm tran}}=  \big( \bT_\Delta[Q]\big)^c $ otherwise. \\

\section{Algebraic contributions}\label{alg-sec}

In this section, we handle the algebraic contributions ${\rm Tang}$ and ${\rm Tran}$. 
We aim to prove the following proposition. 

\begin{proposition}\label{prop-alg}
For $p\geq 2(d+2)$,  we have 
\begin{equation}
{\rm Tang} + {\rm Tran} \lesssim DR^{-\frac{d}{2}+3\delta} \|f\|_2\,. 
\end{equation}
\end{proposition}

The rest of the section is devoted to a proof of this proposition.  

\subsection{Estimates of ${\rm Tang}$}
From the definition of $\bT_\Xi[Q]$ with $Q\in \mathcal Q^*_\Xi$, we see that for those rectangles $T$ with $T^*\cap W\cap Q\neq \emptyset$, $T\cap Q$ is contained in $B^2(5R^d/\Xi)\cap {\mathcal N}_{5R^{1+\delta}}Z(P)$.  We only need to focus on those 
$T$ with $T^*\cap W\cap Q\neq \emptyset$, since otherwise $T$ contributes a negligible rapid decay of $R$  in the integral over $W\cap Q$.  
We use $T_{\Xi, Q}$ to denote a rectangle in $Q$ such that
\begin{itemize}
\item  it contains $T\cap Q$\,,
\item  its long side is $2^{k^*}/\Xi\sim R^d/\Xi$\,,
\item  its short side equals to the short side of $T$.
\end{itemize}
Then we see that only those $T_{\Xi, Q}$'s lying in  $B^2(5R^d/\Xi)\cap {\mathcal N}_{5R^{1+\delta}}Z(P)$ make 
significant contributions in the integral.   
Applying Lemma \ref{lem-geo}, we get
\begin{equation}
 {\rm Cardinality} (\{e(T_{\Xi, Q}): T^*\cap W\cap Q\neq \emptyset; T\in \bT_\Xi[Q]\}) \lesssim 
 DR^{2\delta}. 
\end{equation}
Henceforth,  we estimate ${\rm Tang}$ by 
\begin{equation}
DR^{3\delta} \bigg( \sum_{\omega} \int \big|f_\omega\big|^p\bigg)^{1/p}\,,
\end{equation}
because for given $Q$, there are at most $O(DR^\delta)$ many $\omega$'s making contributions. 
Now apply Bernstein's inequality to obtain 
\begin{equation}\label{tang-est}
{\rm Tang} \lesssim DR^{3\delta} R^{\frac{d+1}{p}-\frac{d+1}{2}} \big(\sum_{\omega}\|f_\omega\|_2^p \big)^{1/p}
\lesssim DR^{3\delta} R^{-d/2} \|f\|_2\,,
\end{equation}
because $p\geq 2(d+1) >2$. 

\subsection {Estimates of {\rm {Tran}}}

For any $Q\in \mathcal Q^{**}_\Delta$, we see that its side-length is $\frac{2^{k^*}}{R^{d-2}\Delta}\sim \frac{R^2}{\Delta}$. 
When $T\in \bT_{\Delta, \rm tran}[Q]$, by the definition of $\bT_\Delta[Q]$,  the rectangle $T$ stays inside $B^2(5R^2/\Delta)
\cap \mathcal N_{R^{1+\delta}}Z(P)$.  Restricted in a ball of $R^2/\Delta$, the thin rectangle $T$ becomes a fatter one $\overline T$, whose dimensions is $\frac{R^2}{\Delta} \times R$.  Note that if $e(T)$ and $e(T')$ is close to each other, say
$|e(T)\wedge e(T')|\leq  \frac{\Delta}{R}$, then $\overline T$ essentially coincides with $\overline {T'}$.  Thus we only have 
at most $R/\Delta$ many directions for those $\overline T$'s. Since they lie in $B^2(5R^2/\Delta)
\cap \mathcal N_{R^{1+\delta}}Z(P)$, it follows from  Lemma \ref{lem-geo} that the cardinality of $\{e(\overline T)\}$ is at most 
$cDR^{\delta}$.  Henceforth we obtain the following geometric lemma. 

\begin{lemma}\label{geo-tr}
Given $Q\in \mathcal Q^{**}_\Delta$, the cardinality of $\{e(T): T\in \bT_{\Delta, \rm tran}[Q]\}$ is at most 
$cDR^{\delta} \Delta$.
\end{lemma}

Furthermore, when $T\in  \bT_{\Delta, \rm tran}[Q]$ and $T^*\cap W\cap Q\neq\emptyset$, 
the rectangle $T$ is transverse to $Z(P)$ at some point $z\in 3Q$ with an angle at least $\f12 \Delta R^{-1}$. In order to use this transversal condition, we need a lemma of Guth as follows.

\begin{lemma}[Guth]
Suppose that 
\begin{itemize}
\item $T$ is a rectangle in $\mathbb R^2$ with width $2\rho$ and arbitrary length.
\item $a\in (0, 1/10)$ denotes an angle.
\item $T$ is subdivided into rectangular segments of length $\geq \rho a^{-1}$.
\item $P$ is a non-singular polynomial of degree $D$.
\item $Z_{\geq a}[P]:= \{z\in Z(P): {\rm Angle}\big(e(T), T_z [Z(P)]\big)\geq a\}$. 
\end{itemize}
Then $Z_{\geq a}[P]\cap T$ is contained in $\lesssim D^2$ of the tube segments of $T$. 
\end{lemma}

\begin{lemma}\label{lem-tran-1}
Given a rectangle $T$, the cardinality of $\{Q\in \mathcal Q^{**}_\Delta: Q\cap W\cap T^*\neq \emptyset, T\in \bT_{\Delta, \rm tran}[Q]\}$ is at most $CD^2$. 
\end{lemma}
\begin{proof}
It follows from Guth's lemma, Lemma \ref{lem-tran-1},  by taking $a= \Delta/R$ and $\rho=R$.  
\end{proof}

Intuitively, Lemma \ref{lem-tran-1} says that a rectangle $T$ crosses at most $\sim D^2$ many cubes $Q$'s
in which it is transverse to an algebraic variety.    
We denote $\bT_{\Delta, {\rm tran}}[Q]\cap\{T: T^*\cap W\cap Q\neq\emptyset\}$ by $\bT^*_{\Delta, {\rm tran}}[Q]$. 

\begin{lemma}
Given $Q\in \mathcal Q^{**}_\Delta$ and an $R$-cube $Q_0\subset Q$,  we have, up to a negligible rapid decay, 
\begin{equation}\label{p-est}
\bigg(  \int_{W\cap Q_0}  \big| \sum_{T\in \bT_{\Delta, {\rm tran}}[Q] } f_T\big|^p\bigg)^{\frac{1}{p}}
\lesssim  D^{\f12} \big( \frac{\Delta}{R}\big)^{\f12-\frac{3}{p}} R^{-\f12+\delta} \big\| \sum_{T\in\bT^*_{\Delta, {\rm tran}}[Q] }f_T\big\|_{L^2{(Q_0)}}\,.
\end{equation}
\end{lemma}

\begin{proof}
Because Fourier transform of $\sum_{T\in \bT_{\Delta, {\rm tran}}[Q] } f_T(x)$ is supported is an $R^{-1}$-neighborhood 
of $\Gamma =\{(\xi, \xi^d):  |\xi|\sim 1\}$, we can apply the Stein-Tomas theorem to get
\begin{equation}\label{6-est}
\bigg(  \int_{W\cap Q_0}  \big| \sum_{T\in \bT_{\Delta, {\rm tran}}[Q] } f_T\big|^6\bigg)^{\frac{1}{6}}
\lesssim R^{-\f12+\delta} \big\| \sum_{T\in\bT^*_{\Delta, {\rm tran}}[Q] }f_T\big\|_{L^2{(Q_0)}}\,.
\end{equation}
On the other hand, by Lemma {\ref{geo-tr}}, we see that Fourier support of $\sum_{T\in\bT^*_{\Delta, {\rm tran}}[Q] }f_T$
has its Lebesgue measure bounded by $DR^{-1+\delta}\frac{\Delta}{R}$.  Henceforth, 
\begin{equation}\label{inf}
\big\|\sum_{T\in\bT^*_{\Delta, {\rm tran}}[Q] }f_T\big\|_{L^\infty(Q_0)} \lesssim 
D^{\f12}  R^{-\f12+\frac{\delta}{2}} \big(\frac{\Delta}{R}\big)^{\f12} \big\| \sum_{T\in\bT^*_{\Delta, {\rm tran}}[Q] }f_T\big\|_{L^2{(Q_0)}}\,.
\end{equation}
Interpolating (\ref{6-est}) and (\ref{inf}), we end up with (\ref{p-est}), as desired. 
\end{proof}

We now turn to estimate the transverse term ${\rm Tran }$. We break each $Q\in \mathcal Q^{**}_\Delta$
into dyadic $R$-cubes $Q_0$ and then we have 
$$
 \bigg( \sum_{Q\in {\mathcal Q}^{**}_{\Delta} } \int_{W\cap Q}  \big| \sum_{T\in \bT_{\Delta, {\rm tran}}[Q] } f_T\big|^p\bigg)^{\frac{1}{p}} =
 \bigg( \sum_{Q\in {\mathcal Q}^{**}_{\Delta} } \sum_{Q_0: Q_0\subset Q}\int_{W\cap Q_0}  \big| \sum_{T\in \bT_{\Delta, {\rm tran}}[Q] } f_T\big|^p\bigg)^{\frac{1}{p}} \,,$$
which is bounded by, via a use of (\ref{p-est}), 
\begin{eqnarray*}
   & & D^{\f12} \big( \frac{\Delta}{R}\big)^{\f12-\frac{3}{p}} R^{-\f12+\delta}
\bigg( \sum_{Q\in {\mathcal Q}^{**}_{\Delta} }\sum_{Q_0: Q_0\subset Q} \big\| \sum_{T\in\bT^*_{\Delta, {\rm tran}}[Q] }f_T\big\|^p_{L^2{(Q_0)}} \bigg)^{\frac{1}{p} }\\
&\leq & D^{\f12} \big( \frac{\Delta}{R}\big)^{\f12-\frac{3}{p}} R^{-\f12+\delta}
\bigg(  \sum_{Q\in {\mathcal Q}^{**}_{\Delta} } \big\| \sum_{T\in\bT^*_{\Delta, {\rm tran}}[Q] }f_T\big\|^2_{L^2{(Q)}} \bigg)^{\frac{1}{p}} \sup_{Q_0} \big\|  \sum_{T\in\bT^*_{\Delta, {\rm tran}}[Q] }f_T\big\|_{L^2(Q_0)}^{1-\frac{2}{p}}\,.
\end{eqnarray*}
By $L^2$-orthogonality, we have 
\begin{equation}\label{loc-1}
 \big\|  \sum_{T\in\bT^*_{\Delta, {\rm tran}}[Q] }f_T\big\|_{L^2(Q_0)} 
 \leq  \bigg(\sum_{T}\big\| f_T\big\|^2_{L^2(Q_0)}  \bigg)^{\f12} \leq R^{-\frac{d-1}{2}} \|f\|_2\,. 
 \end{equation}
Employing $L^2$-orthogonality again and Lemma \ref{lem-tran-1}, we get
\begin{equation}
\bigg(  \sum_{Q\in {\mathcal Q}^{**}_{\Delta} } \big\| \sum_{T\in\bT^*_{\Delta, {\rm tran}}[Q] }f_T\big\|^2_{L^2{(Q)}} \bigg)^{\frac{1}{p}} \lesssim D^{2/p} R^{-\frac{d-1}{p}} \big( \frac{R}{\Delta}\big)^{\frac{1}{p}} \|f\|_2\,. 
\end{equation}
Putting those inequalities above together with the fact that there are about $\sim \log R$ many dyadic $\Delta$'s, we obtain 
\begin{equation}\label{tran-est}
{\rm Tran}\lesssim  D \big( \frac{\Delta}{R}\big)^{\f12-\frac{4}{p}} R^{-\f12+2\delta}R^{-\frac{d-1}{2}} \|f\|_2
\lesssim D R^{-\frac{d}{2}+2\delta} \|f\|_2\,,
\end{equation}
since $p\geq 2(d+1)\geq 8$ and $\Delta\leq R$. \\

It is clear now that Proposition \ref{prop-alg} is a consequence of (\ref{tang-est}) and (\ref{tran-est}). \\

\end{document}